\documentclass[12pt, a4paper, oneside]{amsart}
\usepackage{amsmath, amsthm, amssymb, geometry, mathrsfs, hyperref, tikz-cd}
\usepackage[shortlabels, inline]{enumitem}

\newtheorem{theorem}{Theorem}[section]

\newtheorem{lemma}[theorem]{Lemma}
\newtheorem{corollary}[theorem]{Corollary}

\theoremstyle{definition}
\newtheorem{definition}[theorem]{Definition}
\newtheorem{question}[theorem]{Question}

\newtheorem{example}[theorem]{Example}
\newtheorem{remark}[theorem]{Remark}

\newcommand\B{\mathbb{B}}
\newcommand\C{\mathbb{C}}
\newcommand\D{\mathbb{D}}

\newcommand\N{\mathbb{N}}

\newcommand\R{\mathbb{R}}
\newcommand\Z{\mathbb{Z}}

\newcommand{\cC}{\mathcal{C}}

\newcommand{\cO}{\mathcal{O}}

\newcommand{\id}{\mathrm{id}}
\newcommand{\pr}{\mathrm{pr}}

\newcommand{\T}{\mathrm{T}}

\newcommand{\J}{\mathrm{J}}
\renewcommand{\j}{\mathrm{j}}

\DeclareMathOperator{\Op}{Op}

\newcommand{\Ell}{\mathrm{Ell}}

\begin{document}

\title[Oka properties of complements of holomorphically convex sets]{Oka properties of complements of holomorphically convex sets}
\author{Yuta Kusakabe}
\address{Department of Mathematics, Graduate School of Science, Kyoto University, Kyoto 606-8502, Japan}
\email{kusakabe@math.kyoto-u.ac.jp}

\begin{abstract}
Our main theorem states that the complement of a compact holomorphically convex set in a Stein manifold with the density property is an Oka manifold.
This gives a positive answer to the well-known long-standing problem in Oka theory whether the complement of a compact polynomially convex set in $\C^{n}$ $(n>1)$ is Oka.
Furthermore, we obtain new examples of nonelliptic Oka manifolds which negatively answer Gromov's question.
The relative version of the main theorem is also proved.
As an application, we show that the complement $\C^{n}\setminus\R^{k}$ of a totally real affine subspace is Oka if $n>1$ and $(n,k)\neq(2,1),(2,2),(3,3)$.
\end{abstract}

\maketitle

%
%

\section{Introduction}

The aim of this paper is to solve the long-standing problem which asks whether the complement of a compact polynomially convex set in $\C^{n}$ $(n>1)$ is an Oka manifold \cite[Problem 3.11]{Forstneric2003d}, \cite[Problem]{Forstneric2001a}.
This yields new examples of nonelliptic Oka manifolds which negatively answer Gromov's question \cite[Question 3.2.A$''$]{Gromov1989}.

The notion of Oka manifolds was born in the last few decades, but its roots lie in the classical work of Oka \cite{Oka1939} on the second Cousin problem, that is, the problem of finding a holomorphic function with prescribed zeros.
Oka proved that this problem is solvable on a domain of holomorphy in $\C^{n}$ $(n\geq1)$ if it has a continuous solution, which is now called the Oka principle.
Equivalently, this Oka principle states that holomorphic line bundles on such a domain are holomorphically isomorphic if they are topologically isomorphic.
This was generalized by Grauert \cite{Grauert1957,Grauert1958b} to holomorphic principal bundles over a Stein manifold.
Also in Grauert's Oka principle, the essential point is that every continuous section of a holomorphic fiber bundle with a complex homogeneous fiber can be deformed into a holomorphic section \cite{Grauert1957,Grauert1958b,Grauert1963,Ramspott1965}.
Therefore, the Oka principle can also be viewed as the homotopy principle in complex analysis.

From the viewpoint of the homotopy principle, Gromov \cite{Gromov1986,Gromov1989} initiated modern Oka theory in the late 1980s.
In his seminal paper \cite{Gromov1989}, Gromov introduced the notion of elliptic manifolds which generalizes complex homogeneous manifolds.
A complex manifold $Y$ is said to be \emph{elliptic} if there exists a holomorphic map $s:E\to Y$ from a holomorphic vector bundle $E$ on $Y$ such that $s(0_{y})=y$ and $s|_{E_{y}}:E_{y}\to Y$ is a submersion at $0_{y}$ for each $y\in Y$.
Gromov's main result in \cite{Gromov1989} was that every elliptic manifold $Y$ enjoys the so-called parametric Oka property\footnote{The condition $\Ell_\infty$ in \cite[Remark 3.2.A]{Gromov1989} is equivalent to this Oka property (cf. \cite[\S 5.15]{Forstneric2017}).}:

\emph{
  For any Stein manifold $X$, any closed complex subvariety $X'\subset X$, any compact $\cO(X)$-convex set $K\subset X$, any pair of compact sets $P_{0}\subset P$ in a Euclidean space and any (continuous) family of continuous maps $f_{0}:P\times X\to Y$ such that $f_{0}|_{P_{0}\times X}$, $f_{0}|_{P\times X'}$ and $f_{0}|_{P\times K}$ are families of holomorphic maps, there exists a homotopy $f_{t}:P\times X\to Y$ $(t\in[0,1])$ such that the following hold for each $t\in[0,1]$:
  \begin{enumerate}
  \item $f_{t}=f_{0}$ on $(P_{0}\times X)\cup(P\times X')$,
  \item $f_{t}|_{P\times K}$ is a family of holomorphic maps which approximates $f_{0}$ uniformly on $P\times K$, and
  \item $f_{1}:P\times X\to Y$ is a family of holomorphic maps.
  \end{enumerate}
}

The relative version of this Oka principle was also established in \cite{Gromov1989}, thereby generalizing the classical Oka principle described above.
In the same paper, Gromov asked whether a Runge type approximation property on a certain class of compact sets implies the above parametric Oka property \cite[3.4.(D)]{Gromov1989}, which was finally solved by Forstneri\v{c} in 2009 \cite{Forstneric2006, Forstneric2009}.
This implies that several ostensibly different Oka properties are equivalent \cite[\S 5.15]{Forstneric2017}, and a complex manifold enjoying one (and hence all) of the Oka properties is called an Oka manifold.

\begin{definition}[{\cite{Forstneric2009}}]
  \label{definition:Oka}
  A complex manifold $Y$ is an \emph{Oka manifold} if any holomorphic map from an open neighborhood of a compact convex set $K\subset\C^n\ (n\in\N)$ to $Y$ can be uniformly approximated on $K$ by holomorphic maps $\C^n\to Y$.
\end{definition}

The theory of Oka manifolds has many applications (cf. \cite{Eliashberg1992,Ivarsson2012a,Schurmann1997} and \cite[Part III]{Forstneric2017}) as the homotopy principle does in differential topology (cf. \cite{Gromov1986}).
Using this theory, Eliashberg--Gromov \cite{Eliashberg1992} and Sch\"{u}rmann \cite{Schurmann1997} solved Forster's conjecture on the optimal dimensions of proper holomorphic embeddings of Stein manifolds.
It is known that the construction of proper holomorphic maps is closely related to the Oka property of the complement of a compact holomorphically convex set \cite{Forstneric2019,Forstneric2014}.
Moreover, it is important to verify the Oka property of such a complement also because of its potential \cite{Andrist2016} to be a counterexample to Gromov's question on the implication from the Oka property to ellipticity \cite[Question 3.2.A$''$]{Gromov1989}.
However, this Oka property has remained an open problem for decades since it was asked in \cite[Problem 3.11]{Forstneric2003d}, \cite[Problem]{Forstneric2001a} at the beginning of the 21st century.

Our main theorem is the following, which solves this long-standing problem.

\begin{theorem}
  \label{theorem:main}
  For any Stein manifold $Y$ with the density property and any compact $\cO(Y)$-convex set $K\subset Y$, the complement $Y\setminus K$ is Oka.
\end{theorem}

Here, a complex manifold $Y$ is said to have the \emph{density property} if the Lie algebra generated by all $\C$-complete holomorphic vector fields on $Y$ is dense in the Lie algebra of all holomorphic vector fields on $Y$.
The most typical examples are $\C^{n}$ $(n>1)$\footnote{In fact, there are no 1-dimensional Stein manifolds with the density property.}, which is a consequence of the classical Anders\'{e}n--Lempert theory \cite{Andersen1992} (see \cite[\S4.10]{Forstneric2017} and \cite{Forstneric2022}).
Thus, Theorem \ref{theorem:main} solves the above problem as follows.

\begin{corollary}
  \label{corollary:polynomial}
  For any compact polynomially convex set $K\subset\C^{n}$ $(n>1)$, the complement $\C^{n}\setminus K$ is Oka.
\end{corollary}

The only examples of closed sets in $\C^n$ which have Oka complements so far were closed sets of positive codimension (closed complex subvarieties of codimension at least two \cite{Flenner2016,Forstneric2002,Gromov1989,Kusakabe2021a,Kusakabe2021b,Kusakabe2021}, closed complex hypersurfaces \cite{Hanysz2014,Kusakabe2021a} and closed countable sets of codimension at least two \cite{Kusakabe2020,Winkelmanna}).
Corollary \ref{corollary:polynomial} gives whole new examples of arbitrary codimension.
It also refines a restricted version of the Oka property proved by Forstneri\v{c} and Ritter \cite[Theorem 1]{Forstneric2014} (cf. \cite[Remark 1.3]{Forstneric2019}).
We note that, after this paper has appeared as a preprint, Forstneri\v{c} and Wold \cite{Forstneric2020a} obtained another proof of Theorem \ref{theorem:main} (see \cite{Forstnerica} for further developments).

Corollary \ref{corollary:polynomial} has the following immediate consequence since every $n$-dimensional connected Oka manifold is the image of a holomorphic map from $\C^{n}$ by the result of Forstneri\v{c} \cite[Theorem 1.1]{Forstneric2017a}.

\begin{corollary}
For any compact polynomially convex set $K\subset\C^{n}$ $(n>1)$, there exists a holomorphic map $f:\C^{n}\to\C^{n}$ such that $f(\C^{n})=\C^{n}\setminus K$.
\end{corollary}

As another application of Theorem \ref{theorem:main}, we obtain new examples of nonelliptic Oka manifolds.
Gromov's main result in \cite{Gromov1989} was the implication from ellipticity to the Oka property.
In the same paper, he also proved the converse for Stein manifolds \cite[Remark 3.2.A]{Gromov1989}.
Then he asked a question whether the converse holds for \emph{all} complex manifolds \cite[Question 3.2.A$''$]{Gromov1989}.
This question was negatively answered in our previous paper by the nonelliptic Oka complement of a non-discrete compact countable set in $\C^{n}$ $(n\geq3)$ \cite[Corollary 1.4]{Kusakabe2020}.
The result of Andrist, Shcherbina and Wold \cite[Theorem 1.1]{Andrist2016} (see also \cite[Lemma 3.3]{Kusakabe2020}) and Theorem \ref{theorem:main} give the following new examples of nonelliptic Oka manifolds.

\begin{corollary}
\label{corollary:Gromov}
Let $Y$ be a Stein manifold with the density property and $K\subset Y$ be a compact $\cO(Y)$-convex set.
Assume that $\dim Y\geq 3$ and $K\subset Y$ has infinitely many accumulation points.
Then the complement $Y\setminus K$ is a nonelliptic Oka manifold.
\end{corollary}

The proof of Theorem \ref{theorem:main} is given in Section \ref{section:proof}.
In fact, we show the relative version (Theorem \ref{theorem:main_relative}) of Theorem \ref{theorem:main}.
In Section \ref{section:application}, we prove the following generalization of Corollary \ref{corollary:polynomial} by using this relative version.
Here, the polynomial hull $\widehat S$ of a closed set $S\subset\C^{n}$ is defined by $\widehat S=\bigcup_{j\in\N}\widehat S_{j}$ where $S=\bigcup_{j\in\N}S_{j}$ is an exhaustion of $S$ by compact sets.

\begin{theorem}
\label{theorem:tame}
Let $S\subset\C^{n}$ $(n>1)$ be a closed polynomially convex set (i.e. $S=\widehat S$).
Assume that for some $C>0$ there exists a holomorphic automorphism $\varphi$ of $\C^{n}$ such that
\begin{align*}
\varphi(S)\subset\left\{(z,w)\in\C^{n-2}\times\C^{2}:\|w\|\leq C(1+\|z\|)\right\}.
\end{align*}
Then the complement $\C^{n}\setminus S$ is Oka.
\end{theorem}

It is also a well-known problem whether the complement of a totally real affine subspace $\R^{k}\subset\C^{n}$ is Oka for $n>1$ and $1\leq k\leq n$ \cite[Problem 1.5]{Forstneric2015}.
As in the case of compact polynomially convex sets, Forstneri\v{c} and Wold proved that for $n>1$ and $1\leq k\leq n-1$ the complement $\C^{n}\setminus\R^{k}$ enjoys a restricted version of the Oka property \cite[Theorem 1.3]{Forstneric2015} (see also \cite{Kutzschebauch2018}).
Observe that
\begin{align*}
\C^{n}\setminus\R^{2k}&\cong\C^{n}\setminus\left\{(z,0,\ldots,0,\bar z):z\in\C^{k}\right\}, \\
\C^{n}\setminus\R^{2l+1}&\cong\C^{n}\setminus\left(\R\times\left\{(z,0,\ldots,0,\bar z):z\in\C^{l}\right\}\right),
\end{align*}
and the sets
\begin{align*}
\left\{(z,0,\ldots,0,\bar z):z\in\C^{k}\right\},\quad\R\times\left\{(z,0,\ldots,0,\bar z):z\in\C^{l}\right\}
\end{align*}
are contained in $\{(z,w)\in\C^{n-2}\times\C^{2}:\|w\|\leq(1+\|z\|)\}$ if $\max\{k,l+1\}\leq n-2$.
Note that the only exceptions are $\C^{2}\setminus\R$, $\C^{2}\setminus\R^{2}$ and $\C^{3}\setminus\R^{3}$.
Since every closed set $S$ in $\R^{n}\subset\C^{n}$ is polynomially convex (cf. \cite[p.\,3]{Stout2007}), the above observation and Theorem \ref{theorem:tame} imply the following corollary (see \cite[Proposition 4.9]{Forstnerica} for the case $(n,k)=(2,1)$).

\begin{corollary}
If $n>1$, then for any compact set $K$ in $\R^{n}\subset\C^{n}$ the complement $\C^{n}\setminus K$ is Oka.
If in addition $(n,k)\neq(2,1),(2,2),(3,3)$, then for any closed set $S$ in $\R^{k}\subset\C^{n}$ the complement $\C^{n}\setminus S$ is Oka.
In particular, the complement $\C^{n}\setminus\R^{k}$ of a totally real affine subspace is Oka if $n>1$ and $(n,k)\neq(2,1),(2,2),(3,3)$.
\end{corollary}

In the above results, we only consider the complements of holomorphically convex sets.
It can be easily seen that a domain in a complex manifold with a strongly pseudoconvex boundary point cannot be Oka (cf. \cite[Example 5]{Forstneric2014}).
Thus it is natural to ask whether the Oka property of the complement $\C^n\setminus S$ of a closed set $S$ implies polynomial convexity of $S$.
In fact, we can obtain the following negative answer to this question as an application of Theorem \ref{theorem:tame}.

\begin{corollary}
\label{corollary:scc}
For any rectifiable simple closed curve $C\subset\C^{n}$ $(n\geq3)$, the complement $\C^{n}\setminus C$ is Oka.
In particular, the complement $\C^{n}\setminus S^{1}$ of the unit circle $S^{1}$ in a complex line $\C\subset\C^{n}$ $(n\geq 3)$ is Oka.
\end{corollary}

The proof of Corollary \ref{corollary:scc} is given in Section \ref{section:application}.
In the same section, we give further applications of our results.

%
%

\section{Dominating sprays and the Oka property}
\label{section:spray}

In this section, we recall the notion of dominating sprays which plays a fundamental role in the proof of Theorem \ref{theorem:main}.
For a subset $A$ of a topological space $X$, let $\Op A$ denote a non-specified open neighborhood of $A$ in $X$.

\begin{definition}
\label{definition:spray}
Let $X$ be a (reduced) complex space, $\pi:Y\to B$ be a holomorphic submersion (between complex spaces) and $A\subset X$ be a subset.
\begin{enumerate}[leftmargin=*]
\item A \emph{(local) $\pi$-spray} over a holomorphic map $f:\Op A\to Y$ is a holomorphic map $s:\Op A\times W\to Y$ where $W\subset\C^N$ is an open neighborhood of $0$ such that $s(x,0)=f(x)$, $\pi\circ s(x,w)=\pi\circ f(x)$ for all $(x,w)\in\Op A\times W$.
Particularly in the case of $W=\C^{N}$, $s$ is also called a \emph{global $\pi$-spray}.
\item A $\pi$-spray $s:\Op A\times W\to Y$ is \emph{dominating} if $s(x,\cdot):W\to\pi^{-1}(\pi\circ s(x,0))$ is a submersion at $0$ for each $x\in A$.
\end{enumerate}
\end{definition}

The following variant of Gromov's ellipticity was introduced in our previous paper \cite{Kusakabe2021b}.

\begin{definition}
A holomorphic submersion $\pi:Y\to B$ is \emph{convexly elliptic} if there exists an open cover $\{U_{\alpha}\}_{\alpha}$ of $B$ such that for any compact convex set $K\subset\C^{n}$ $(n\in\N)$ and any holomorphic map $f:\Op K\to Y$ with $f(K)\subset \pi^{-1}(U_{\alpha})$ for some $\alpha$ there exists a dominating global $\pi$-spray over $f$.
\end{definition}

A holomorphic submersion is said to enjoy \emph{the Oka property} if it enjoys one (and hence all) of the equivalent Oka properties in \cite[Corollary 5.5]{Kusakabe2021b}.
For example, a holomorphic submersion $\pi:Y\to B$ enjoys the Oka property if and only if for any bounded convex domain $\Omega$, any compact convex set $K\subset\Omega\times\C^n$ $(n\in\N)$, any holomorphic map $F:\Omega\to B$ and any continuous map $f:\Omega\times\C^{n}\to Y$ such that $f|_{\Op K}$ is holomorphic and $\pi\circ f(z,w)=F(z)$ for all $(z,w)\in\Omega\times\C^n$ there exists a holomorphic map $\tilde f:\Omega\times\C^{n}\to Y$ which approximates $f$ uniformly on $K$ and satisfies $\pi\circ\tilde f(z,w)=F(z)$ for all $(z,w)\in\Omega\times\C^n$ (cf. \cite[Definition 5.1]{Kusakabe2021b}).
Thus a holomorphic submersion $Y\to*$ onto the singleton enjoys the Oka property if and only if $Y$ is Oka in the sense of Definition \ref{definition:Oka}.
The following characterization of the Oka property is the main theorem in \cite{Kusakabe2021b}.

\begin{theorem}[{cf. \cite[Theorem 2.2]{Kusakabe2021a} and \cite[Theorem 1.3 and Corollary 5.5]{Kusakabe2021b}}]
\label{theorem:elliptic_characterization}
A holomorphic submersion enjoys the Oka property if and only if it is convexly elliptic.
\end{theorem}

In the proof of Theorem \ref{theorem:main}, we also need the following fact which ensures the existence of biholomorphic local sprays.
Here, $\B^{N}$ denotes the open unit ball in $\C^{N}$.

\begin{lemma}[{cf. \cite[Lemma 5.10.4]{Forstneric2017}}]
\label{lemma:tubular_neighborhood}
Assume that $K\subset\C^{n}$ is a compact convex set, $\pi:Y\to\Op K$ is a holomorphic submersion and $f:\Op K\to Y$ is a holomorphic section of $\pi$.
Then there exists a biholomorphic local $\pi$-spray $\iota:\Op K\times\B^{N}\to\iota(\Op K\times\B^{N})\subset Y$ over $f$.
\end{lemma}

%
%

\section{The density property}
\label{section:density}

In this section, we recall the definition of the density property for holomorphic submersions and the relative Anders\'{e}n--Lempert theorem to prove the relative version (Theorem \ref{theorem:main_relative}) of Theorem \ref{theorem:main}.
The following definition was introduced by Andrist and Kutzschebauch \cite{Andrist2018}.

\begin{definition}[{Andrist--Kutzschebauch \cite[Definition 1.2]{Andrist2018}}]
A holomorphic submersion $\pi:Y\to B$ enjoys the \emph{(fibered) density property} if the Lie algebra generated by all $\C$-complete holomorphic vector fields on $Y$ tangent to the fibers of $\pi$ is dense in the Lie algebra of all holomorphic vector fields on $Y$ tangent to the fibers of $\pi$ with respect to the compact-open topology.
\end{definition}

\begin{example}
\label{example:density}
Assume that $Y$ is a Stein manifold with the density property and $B$ is a Stein space.
Then the trivial bundle $B\times Y\to B$ enjoys the density property by \cite[Lemma 3.5]{Varolin2001}.
\end{example}

We can easily see that the density property is stable under pullback.

\begin{lemma}
\label{lemma:pullback}
Assume that $X$ and $Y$ are Stein spaces, $\pi:Y\to B$ is a holomorphic submersion with the density property and $f:X\to B$ is a holomorphic map.
Then the pullback submersion $f^{*}\pi:f^{*}Y\to X$ enjoys the density property.
\end{lemma}

\begin{proof}
Note that we have a pullback diagram:
\begin{center}
\begin{tikzcd}
f^{*}Y \arrow[d, "f^{*}\pi"] \arrow[rr, "{(f^{*}\pi,\pi^{*}f)}"] & \  & X\times Y \arrow[d, "\id_{X}\times\pi"] \\
X \arrow[rr, "{(\id_{X},f)}"] & & X\times B  \arrow[ul, phantom, "\ulcorner", very near start]
\end{tikzcd}
\end{center}
Let $V$ be a holomorphic vector field on $f^{*}Y$ tangent to the fibers of $f^{*}\pi$.
Since $(f^{*}\pi,\pi^{*}f):f^{*}Y\to X\times Y$ is a proper holomorphic embedding and $X\times Y$ is Stein, there exists a holomorphic vector field $W$ on $X\times Y$ tangent to the fibers of $\id_{X}\times\pi$ such that $(f^{*}\pi,\pi^{*}f)^{*}W=V$.
Then the argument in the proof of \cite[Lemma 3.5]{Varolin2001} implies that $W$ can be approximated by Lie combinations of $\C$-complete holomorphic vector fields on $X\times Y$ tangent to the fibers of $\id_{X}\times\pi$.
Thus the holomorphic vector field $V=(f^{*}\pi,\pi^{*}f)^{*}W$ can also be approximated by Lie combinations of $\C$-complete holomorphic vector fields on $f^{*}Y$ tangent to the fibers of $f^{*}\pi$.
\end{proof}

The following is the relative Anders\'{e}n--Lempert theorem proved by Andrist and Kutzschebauch \cite{Andrist2018}.
It is slightly different from \cite[Theorem 1.4]{Andrist2018}, but it also holds because a holomorphically convex set in a Stein space admits a basis of open Runge neighborhoods.

\begin{theorem}[{cf. Andrist--Kutzschebauch \cite[Theorem 1.4]{Andrist2018}}]
\label{theorem:Andersen-Lempert}
Let $Y$ be a Stein manifold, $\pi:Y\to B$ be a holomorphic submersion with the density property, $U\subset Y$ be an open set and $\varphi_{t}:U\to Y$ $(t\in[0,1])$ be a fiber-preserving $\cC^{1}$-isotopy of injective holomorphic maps such that $\varphi_{0}$ is the inclusion $U\hookrightarrow Y$.
Assume that $K\subset U$ is a compact set such that $\varphi_{t}(K)$ is $\cO(Y)$-convex for every $t\in[0,1]$.
Then there exists a fiber-preserving homotopy $\widetilde\varphi_{t}:Y\to Y$ $(t\in[0,1])$ of holomorphic automorphisms such that $\widetilde\varphi_{0}=\id_{Y}$ and $\widetilde\varphi_{t}$ approximates $\varphi_{t}$ uniformly on $K$ for every $t\in[0,1]$.
\end{theorem}

In the proof of the relative version (Theorem \ref{theorem:main_relative}) of Theorem \ref{theorem:main}, we also need the jet interpolation condition (see \cite{Ramos-Peon2019} for some related results).
Here, a compact subset of a complex space is called a \emph{Stein compact} if it admits a basis of open Stein neighborhoods, and $\j^{k}_{y}\varphi$ denotes the $k$-jet of $\varphi$ at $y$.

\begin{corollary}
\label{corollary:Andersen-Lempert_jet}
Let $\pi:Y\to B$, $K\subset U\subset Y$ and $\varphi_{t}:U\to Y$ $(t\in[0,1])$ be as in Theorem \ref{theorem:Andersen-Lempert}.
Assume that $L\subset B$ is a Stein compact, $f:\Op L\to Y$ is a holomorphic section of $\pi$ such that $f(L)\subset K^{\circ}$, and $\varphi_{t}(K)$ is $\cO(Y)$-convex for every $t\in[0,1]$.
Then for any nonnegative integer $k$ there exists a fiber-preserving homotopy $\widetilde\varphi_{t}:\pi^{-1}(\Op L)\to\pi^{-1}(\Op L)$ $(t\in[0,1])$ of holomorphic automorphisms such that
\begin{enumerate}
\item $\widetilde\varphi_{0}=\id_{\pi^{-1}(\Op L)}$,
\item $\widetilde\varphi_{t}$ approximates $\varphi_{t}$ uniformly on $K\cap\pi^{-1}(\Op L)$ for every $t\in[0,1]$, and
\item $\j^{k}_{f(b)}(\widetilde\varphi_{t}|_{\pi^{-1}(b)})=\j^{k}_{f(b)}(\varphi_{t}|_{\pi^{-1}(b)})$ for all $b\in\Op L$ and $t\in[0,1]$.
\end{enumerate}
\end{corollary}

\begin{proof}
By the argument in the proof of \cite[Lemma 7]{Ramos-Peon2019}, there exist $\C$-complete holomorphic vector fields $V_{j}$ $(j=1,\ldots,N)$ on $Y$ tangent to the fibers of $\pi$ such that for the global $\pi$-spray
\begin{align*}
s:\pi^{-1}(\Op L)\times\C^{N}\to\pi^{-1}(\Op L),\quad s(y,w_{1},\ldots,w_{N})=\psi_{1}^{w_{1}}\circ\cdots\circ\psi_{N}^{w_{N}}(y)
\end{align*}
defined by the flows $\psi_{j}^{w}$ $(w\in\C)$ of $V_{j}$, the map
\begin{align*}
\C^{N}\to\J^{k}_{f(b),*}(\pi^{-1}(b)),\quad w\mapsto\j^{k}_{f(b)}(s(\cdot,w)|_{\pi^{-1}(b)})
\end{align*}
to the space of nondegenerate $k$-jets at $f(b)$ is a submersion at $0$ for each $b\in\Op L$ (cf. \cite[\S 2]{Ramos-Peon2019}).
By Theorem \ref{theorem:Andersen-Lempert}, there exists a fiber-preserving homotopy $\Phi_{t}:Y\to Y$ $(t\in[0,1])$ of holomorphic automorphisms such that for each $t\in[0,1]$,
\begin{itemize}
\item $\Phi_{0}=\id_{Y}$,
\item $\Phi_{t}$ approximates $\varphi_{t}$ uniformly on $K$, and
\item $\Phi_{t}^{-1}\circ\varphi_{t}$ is sufficiently close to the identity map on $\Op f(L)$.
\end{itemize}
Since $L$ is a Stein compact, there exists a homotopy of holomorphic maps $g_{t}:\Op L\to\C^{N}$ $(t\in[0,1])$ close to $0$ such that $g_{0}\equiv 0$ and
\begin{align*}
\j^{k}_{f(b)}(s(\cdot,g_{t}(b))|_{\pi^{-1}(b)})=\j^{k}_{f(b)}(\Phi_{t}^{-1}\circ\varphi_{t}|_{\pi^{-1}(b)})
\end{align*}
for all $b\in\Op L$ and $t\in[0,1]$ (cf. \cite[Lemma 2.5]{Kusakabe2021a}).
If we set
\begin{align*}
\widetilde\varphi_{t}=\Phi_{t}\circ s\circ(\id_{\pi^{-1}(\Op L)},g_{t}\circ\pi):\pi^{-1}(\Op L)\to\pi^{-1}(\Op L)\quad(t\in[0,1]),
\end{align*}
it has the desired properties.
\end{proof}

%
%

\section{Proof of Theorem \ref{theorem:main}}
\label{section:proof}

The goal of this section is to prove the relative version (Theorem \ref{theorem:main_relative}) of Theorem \ref{theorem:main}.
To state this theorem, we introduce the following notion.
Holomorphic convexity of a closed set is defined in the same way as polynomial convexity (see Theorem \ref{theorem:tame}).

\begin{definition}
Let $\pi:Y\to B$ be a holomorphic submersion.
A subset $S\subset Y$ is called a \emph{family of compact holomorphically convex sets} if the restriction $\pi|_{S}:S\to B$ is proper and each point of $B$ admits an open neighborhood $U\subset B$ such that $S\cap\pi^{-1}(U)\subset\pi^{-1}(U)$ is $\cO(\pi^{-1}(U))$-convex.
\end{definition}

The following is the relative version of Theorem \ref{theorem:main}.

\begin{theorem}
\label{theorem:main_relative}
Let $\pi:Y\to B$ be a holomorphic submersion between complex spaces and $S\subset Y$ be a family of compact holomorphically convex sets.
Assume that each point of $B$ admits an open neighborhood $U\subset B$ such that $\pi^{-1}(U)$ is Stein and the restriction $\pi^{-1}(U)\to U$ enjoys the density property.
Then the restriction $\pi|_{Y\setminus S}:Y\setminus S\to B$ enjoys the Oka property.
\end{theorem}

In the proof of Theorem \ref{theorem:main_relative}, we use the notion of \emph{non-autonomous basins with uniform bounds} (cf. \cite{Abbondandolo,Fornaess2016}).
Let $Y$ be a complex manifold, $\iota:\B^{n}\to\iota(\B^{n})\subset Y$ be a biholomorphic map and $\{\varphi_{j}\}_{j\in\N}$ be a sequence of holomorphic automorphisms of $Y$.
Assume that $\varphi_{j}\circ\iota(\B^{n})\subset\iota(\B^{n})$ for all $j\in\N$ and there exist $0<C<D<1$ such that
\begin{align*}
\varphi_{j}\circ\iota(0)=\iota(0),\quad C\|w\|\leq\|\iota^{-1}\circ\varphi_{j}\circ\iota (w)\|\leq D\|w\|
\end{align*}
for all $j\in\N$ and $w\in\B^{n}$.
Then the set
\begin{align*}
\left\{y\in Y:\lim_{j\to\infty}\varphi_{j}\circ\cdots\cdot\varphi_{1}(y)=\iota(0)\right\}
\end{align*}
is called a \emph{non-autonomous basin with uniform bounds}.
The long-standing Bedford conjecture states that such a basin is biholomorphic to $\C^{n}$ (cf. \cite{Abbondandolo}).
Forn\ae ss and Wold proved that such a basin is elliptic in the case of $Y=\C^{n}$ \cite[Theorem 1.1]{Fornaess2016}.
In fact, their proof works for any complex manifold $Y$ (see the proof of Lemma \ref{lemma:elliptic_neighborhood}).

\begin{theorem}[{cf. Forn\ae ss--Wold \cite[Theorem 1.1]{Fornaess2016}}]
\label{theorem:non-autonomous}
Every non-autonomous basin with uniform bounds is elliptic.
\end{theorem}

We need the relative version of this theorem.
Let $\pi:Y\to B$ be a holomorphic submersion between complex spaces, $f:B\to Y$ be a holomorphic section of $\pi$ and $\iota:B\times\B^{n}\to\iota(B\times\B^{n})\subset Y$ be a biholomorphic local $\pi$-spray over $f$.
Assume that $\{\varphi_{j}\}_{j\in\N}$ is a sequence of fiber-preserving holomorphic automorphisms of $Y$ such that $\varphi_{j}\circ\iota(B\times\B^{n})\subset\iota(B\times\B^{n})$ for all $j\in\N$ and there exist $0<C<D<1$ such that
\begin{align*}
\varphi_{j}\circ f(b)=f(b),\quad C\|w\|\leq\|\pr_{\B^{n}}\circ\iota^{-1}\circ\varphi_{j}\circ\iota(b,w)\|\leq D\|w\|
\end{align*}
for all $j\in\N$ and $(b,w)\in B\times\B^{n}$ where $\pr_{\B^{n}}:B\times\B^{n}\to\B^{n}$ is the projection to $\B^{n}$.
Let us consider the set
\begin{align*}
\Omega=\left\{y\in Y:\lim_{j\to\infty}\varphi_{j}\circ\cdots\cdot\varphi_{1}(y)\in f(B)\right\}.
\end{align*}

\begin{lemma}
\label{lemma:elliptic_neighborhood}
Under the above assumptions, the restriction $\pi|_{\Omega}:\Omega\to B$ is elliptic.
\end{lemma}

Before proving the above lemma, let us recall the definition of ellipticity for holomorphic submersions.

\begin{definition}
A holomorphic submersion $\pi:Y\to B$ is \emph{elliptic} if there exists an open cover $\{U_{\alpha}\}_{\alpha}$ of $B$ such that for each $\alpha$ there exists a holomorphic vector bundle $p:E\to\pi^{-1}(U_{\alpha})$ and a holomorphic map $s:E\to\pi^{-1}(U_{\alpha})$ such that $\pi\circ s=\pi\circ p$, $s(0_{y})=y$ and $s|_{E_{y}}:E_{y}\to\pi^{-1}(\pi(y))$ is a submersion at $0_{y}$ for each $y\in\pi^{-1}(U_{\alpha})$. 
\end{definition}

It can be easily seen that ellipticity implies convex ellipticity (cf. \cite[Proof of Proposition 8.8.11\,(b)]{Forstneric2017}).

\begin{proof}[Proof of Lemma \ref{lemma:elliptic_neighborhood}]
Since the definition of ellipticity is local, we may assume that $B$ is a locally closed complex subvariety of an affine space.
Take a point $b\in B$ and a compact neighborhood $K\subset B$ of $b$.
It suffices to construct a dominating global $\pi$-spray $s:(\Omega\cap\pi^{-1}(K^{\circ}))\times\C^{n}\to\Omega$ over the inclusion $\Omega\cap\pi^{-1}(K^{\circ})\hookrightarrow\Omega$.
Set $K_{j}=(\varphi_{j}\circ\cdots\circ\varphi_{1})^{-1}(\iota(K\times(1/2)\overline{\B^{n}}))$ $(j\in\N)$, $K_{0}=\iota(K\times(1/2)\overline{\B^{n}})$ and
\begin{align*}
s_{0}:\Op K_{0}\times\Op\,\{0\}^{n}\to\Omega,\quad s_{0}(\iota(b,w),w')=\iota(b,w+w').
\end{align*}
Take $0<r<1$.
For each $j\in\N$, we shall inductively construct a dominating local $\pi$-spray $s_{j}:\Op K_{j}\times\Op(j\overline{\B^{n}})\to\Omega$ over the inclusion $\Op K_{j}\hookrightarrow\Omega$ which approximates $s_{j-1}$ on $\Op K_{j-1}\times\Op(r(j-1)\overline{\B^{n}})$.
Then $s=\lim_{j\to\infty}s_{j}:(\Omega\cap\pi^{-1}(K^{\circ}))\times\C^{n}\to\Omega$ exists and has the desired property.

Assume that we already have $s_{j-1}:\Op K_{j-1}\times\Op((j-1)\overline{\B^{n}})\to\Omega$ for some $j\in\N$.
By definition, there exists $N\in\N$ such that
\begin{align*}
(\varphi_{N}\circ\cdots\circ\varphi_{1})(K_{j}\cup s_{j-1}(K_{j-1}\times(j-1)\overline{\B^{n}}))\subset\iota(K\times\B^{n}).
\end{align*}
Then by the argument in \cite[p.\,4]{Fornaess2016}, there exist a positive continuous function $t:\Op K_{j}\to\R$ and a holomorphic map $\tilde s_{j-1}:\{(y,w)\in\Op K_{j}\times\C^{n}:\|w\|<t(y)\}\to\iota(\Op K\times\B^{n})$ such that
\begin{itemize}
\item $\tilde s_{j-1}|_{\Op K_{j}\times\Op\,\{0\}^{n}}$ is a dominating local $\pi$-spray over $\varphi_{N}\circ\cdots\circ\varphi_{1}|_{\Op K_{j}}$,
\item $t|_{K_{j-1}}>j-1$, and
\item $(\varphi_{N}\circ\cdots\circ\varphi_{1})^{-1}\circ\tilde s_{j-1}$ approximates $s_{j-1}$ on $\Op K_{j-1}\times\Op((j-1)\overline{\B^{n}})$.
\end{itemize}
By applying the argument in the proof of \cite[Lemma 2.2]{Fornaess2016} to
\begin{align*}
(\varphi_{N}\circ\cdots\circ\varphi_{1})^{-1}\circ\tilde s_{j-1}:\{(y,w)\in\Op K_{j}\times\C^{n}:\|w\|<t(y)\}\to\Omega
\end{align*}
as in \cite[p.\,4]{Fornaess2016}, we can obtain a dominating local $\pi$-spray $s_{j}:\Op K_{j}\times\Op(j\overline{\B^{n}})\to\Omega$ over the inclusion which approximates $s_{j-1}$ on $\Op K_{j-1}\times\Op(r(j-1)\overline{\B^{n}})$.
\end{proof}

\begin{proof}[Proof of Theorem \ref{theorem:main_relative}]
By Theorem \ref{theorem:elliptic_characterization}, it suffices to prove that the restriction $\pi|_{Y\setminus S}:Y\setminus S\to B$ is convexly elliptic.
Since the definition of convex ellipticity is local, we may assume that
\begin{itemize}
\item $Y$ is Stein,
\item $\pi:Y\to B$ enjoys the density property, and
\item $S\subset Y$ is $\cO(Y)$-convex
\end{itemize}
by the assumption in Theorem \ref{theorem:main_relative} and the definition of a family of compact holomorphically convex sets.
Let $K\subset\C^{n}$ $(n\in\N)$ be a compact convex set and $f:\Op K\to Y\setminus S$ be a holomorphic map.
Take a small compact convex neighborhood $\widetilde K\subset\C^{n}$ of $K$ such that $f:\Op\widetilde K\to Y\setminus S$ is defined.
If we set
\begin{align*}
\widetilde Y=(\pi\circ f)^{*}Y,\quad \widetilde\pi=(\pi\circ f)^{*}\pi,\quad g=\pi^{*}(\pi\circ f),\quad \widetilde S=g^{-1}(S),
\end{align*}
then the following is a pullback diagram:
\begin{center}
\begin{tikzcd}
\widetilde Y\setminus\widetilde S \arrow[d, "\widetilde\pi|_{Y\setminus\widetilde S}"] \arrow[r, "g|_{Y\setminus\widetilde S}"] & Y\setminus S \arrow[d, "\pi"] \\
\Op\widetilde K \arrow[r, "\pi\circ f"] & B
\end{tikzcd}
\end{center}
Since we may assume that $\Op\widetilde K$ is Stein, it follows from Lemma \ref{lemma:pullback} that $\widetilde Y$ is a Stein manifold and $\widetilde\pi:\widetilde Y\to\Op\widetilde K$ enjoys the density property.
Note also that $\widetilde S\subset\widetilde Y$ is $\cO(\widetilde Y)$-convex and the restriction $\widetilde\pi|_{\widetilde S}:\widetilde S\to\Op\widetilde K$ is proper.

By the universality of the pullback, there exists a holomorphic section $\tilde f:\Op\widetilde K\to\widetilde Y\setminus\widetilde S$ of $\widetilde\pi$ such that $g\circ\tilde f=f$:
\begin{center}
\begin{tikzcd}
\Op\widetilde K \arrow[drr, bend left=15, "f"] \arrow[ddr, bend right=20, equal] \arrow[dr, dashed, "\tilde f"] & & \\
& \widetilde Y\setminus\widetilde S \arrow[d, "\widetilde\pi|_{Y\setminus\widetilde S}"] \arrow[r, "g|_{Y\setminus\widetilde S}"] & Y\setminus S \arrow[d, "\pi"] \\
& \Op\widetilde K \arrow[r, "\pi\circ f"] & B
\end{tikzcd}
\end{center}
Then there exists a biholomorphic local $\pi$-spray $\iota:\Op\widetilde K\times2\B^{N}\to\iota(\Op \widetilde K\times2\B^{N})\subset\widetilde Y\setminus\widetilde S$ over $\tilde f$ by Lemma \ref{lemma:tubular_neighborhood}.
Note that $\widetilde S|_{\widetilde K}=\widetilde S\cap\widetilde\pi^{-1}(\widetilde K)$ is a compact $\cO(\widetilde Y)$-convex set since $\widetilde S$ is $\cO(\widetilde Y)$-convex and $\widetilde K$ is convex.
Take a small compact $\cO(\widetilde Y)$-convex neighborhood $L\subset\widetilde Y\setminus \iota(\Op\widetilde K\times2\B^{N})$ of $\widetilde S|_{\widetilde K}$.
After rescaling $2\B^{N}$ if necessary, we may assume that $\iota(\widetilde K\times r\overline{\B^{N}})\cup L\subset\widetilde Y$ is $\cO(\widetilde Y)$-convex for all $r\in[0,1]$.
Define a smooth isotopy $\mu_{t}:2\B^{N}\to2\B^{N}$ $(t\in[0,1])$ by $w\mapsto (1-t/2)w$, and consider the fiber-preserving smooth isotopy $\varphi_{t}:\iota(\Op\widetilde K\times2\B^{N})\cup\Op(L)\to\widetilde Y$ $(t\in[0,1])$ defined by
\begin{align*}
\varphi_{t}(y)=\begin{cases}
\iota\circ(\id_{\Op\widetilde K}\times\mu_{t})\circ\iota^{-1}(y)\quad\text{if $y\in\iota(\Op\widetilde K\times2\B^{N})$,} \\
y\quad\text{if $y\in\Op L$.}
\end{cases}
\end{align*}
Take a distance function $d$ on $\widetilde Y$ and set
\begin{align*}
\varepsilon=\inf_{y\in \widetilde S|_{\widetilde K},\ y'\in\widetilde Y\setminus L^{\circ}}d(y,y')>0.
\end{align*}
By Corollary \ref{corollary:Andersen-Lempert_jet}, there exist fiber-preserving holomorphic automorphisms $\widetilde\varphi_{j}:\widetilde\pi^{-1}(\Op K)\to\widetilde\pi^{-1}(\Op K)$ $(j\in\N)$ such that for each $j\in\N$,
$\widetilde\varphi_{j}\circ\iota(\Op K\times\B^{N})\subset\iota(\Op K\times\B^{N})$,
\begin{align*}
\widetilde\varphi_{j}\circ\tilde f(z)=\tilde f(z),\quad \frac{1}{4}\|w\|\leq\|\pr_{\B^{N}}\circ\iota^{-1}\circ\widetilde\varphi_{j}\circ\iota(z,w)\|\leq\frac{3}{4}\|w\|
\end{align*}
for all $(z,w)\in\Op K\times\B^{N}$, and 
\begin{align*}
\sup_{y\in L\cap\widetilde\pi^{-1}(\Op K)}d(\widetilde\varphi_{j}(y),y)<\varepsilon/2^{j}.
\end{align*}
Then by Lemma \ref{lemma:elliptic_neighborhood}, the restriction of $\widetilde\pi|_{\widetilde\pi^{-1}(\Op K)}:\widetilde\pi^{-1}(\Op K)\to\Op K$ to the set
\begin{align*}
\Omega=\left\{y\in\widetilde\pi^{-1}(\Op K):\lim_{j\to\infty}\varphi_{j}\circ\cdots\cdot\varphi_{1}(y)\in\tilde f(\Op K)\right\}
\end{align*}
is elliptic and hence convexly elliptic.
Note that $\tilde f(\Op K)\subset\Omega\subset\widetilde Y\setminus\widetilde S$ holds by construction.
Thus there exists a dominating global $\widetilde\pi$-spray $s:\Op K\times\C^{N}\to\Omega\subset\widetilde Y\setminus\widetilde S$ over $\tilde f|_{\Op K}$.
Then the composition $g\circ s:\Op K\times\C^{N}\to Y\setminus S$ is a dominating global $\pi$-spray over $f:\Op K\to Y\setminus S$.
\end{proof}

\begin{remark}
In fact, as pointed out by Forstneri\v{c} and Wold \cite{Forstneric2020a}, our proof of Theorem \ref{theorem:main_relative} can be modified to obtain a family of Fatou--Bieberbach domains as $\Omega$ (see the proofs of \cite[Lemma 3.3]{Forstneric2020a} and \cite[Theorem 4]{Wold2005}).
\end{remark}

%
%

\section{Applications}
\label{section:application}

In this section, we present applications of our results.
Let us first prove Theorem \ref{theorem:tame} in the introduction.
In the proof, we need the following lemma which is an immediate consequence of \cite[Corollary 4.1]{Kusakabe2021a}.

\begin{lemma}
\label{lemma:span}
Let $Y$ be a complex manifold.
Assume that for any $y\in Y$ there exist complex manifolds $B_{j}$ $(j=1,\ldots, k)$ and holomorphic submersions $\pi_{j}:Y\to B_{j}$ $(j=1,\ldots, k)$ with the Oka property such that $\T_{y}Y=\sum_{j=1}^{k}\T_{y}\pi_{j}^{-1}(\pi_{j}(y))$.
Then $Y$ is an Oka manifold.
\end{lemma}

\begin{proof}
Let $X$ be a Stein manifold, $f:X\to Y$ be a holomorphic map and $x_{0}\in X$ be a point.
Take complex manifolds $B_{j}$ $(j=1,\ldots, k)$ and holomorphic submersions $\pi_{j}:Y\to B_{j}$ $(j=1,\ldots, k)$ with the Oka property such that $\T_{f(x_{0})}Y=\sum_{j=1}^{k}\T_{f(x_{0})}\pi_{j}^{-1}(\pi_{j}(f(x_{0})))$.
Then for each $j=1,\ldots, k$, the Oka property of $\pi_{j}$ implies that there exists a dominating global $\pi_{j}$-spray $s_{j}:X\times\C^{N_{j}}\to Y$ over $f$ (cf. the proof of \cite[Corollary 8.8.7]{Forstneric2017}).
Note that
\begin{align*}
\sum_{j=1}^{k}\partial_{w}|_{w=0}s_{j}(x_{0},w)(\T_{0}\C^{N_{j}})=\sum_{j=1}^{k}\T_{f(x_{0})}\pi_{j}^{-1}(\pi_{j}(f(x_{0})))=\T_{f(x_{0})}Y
\end{align*}
holds.
Thus \cite[Corollary 4.1]{Kusakabe2021a} implies that $Y$ is an Oka manifold.
\end{proof}

\begin{proof}[Proof of Theorem \ref{theorem:tame}]
We may assume that there exists $C>0$ such that
\begin{align*}
S\subset\left\{(z,w)\in\C^{n-2}\times\C^{2}:\|w\|\leq C(1+\|z\|)\right\}
\end{align*}
from the beginning.
Then for any $(n-2)$-dimensional complex linear subspace $V$ of $\C^{n}$ which is sufficiently close to $\{(z,w)\in\C^{n-2}\times\C^{2}:w=0\}$ (in the complex Grassmannian), the restriction $\pi|_{S}:S\to V$ of the orthogonal projection $\pi:\C^{n}\to V$ is a proper map.
Thus for any $z\in\C^{n}$ there exist orthogonal projections $\pi_{j}:\C^{n}\to V_{j}$ to $(n-2)$-dimensional complex linear subspaces $V_{j}\subset\C^{n}$ $(j=1,\ldots,k)$ such that
\begin{itemize}
\item the restrictions $\pi_{j}|_{S}:S\to V_{j}$ are proper, and
\item $\T_{z}\C^{n}=\sum_{j=1}^{k}\T_{z}\pi_{j}^{-1}(\pi_{j}(z))$.
\end{itemize}
Then Theorem \ref{theorem:main_relative} implies that the restrictions $\pi_{j}|_{\C^{n}\setminus S}:\C^{n}\setminus S\to V_{j}$ enjoy the Oka property.
Hence the complement $\C^{n}\setminus S$ is Oka by Lemma \ref{lemma:span}.
\end{proof}

Next, we give the proof of Corollary \ref{corollary:scc} in the introduction.
Let us recall the following localization principle for Oka manifolds.
Here, a subset of $Y$ is said to be \emph{Zariski open} if its complement is a closed complex subvariety.

\begin{theorem}[{cf. \cite[Theorem 1.4]{Kusakabe2021a}}]
\label{theorem:localization}
Let $Y$ be a complex manifold.
Assume that each point of $Y$ has a Zariski open Oka neighborhood.
Then $Y$ is an Oka manifold.
\end{theorem}

\begin{proof}[Proof of Corollary \ref{corollary:scc}]
Take an arbitrary point $p\in\C^{n}\setminus C$.
After some unitary change of coordinates, we may assume that
\begin{itemize}
\item there exists a point $q\in C$ such that $(\{0\}\times\C^{n-1})\cap C=\{q\}$, and
\item $p\in\C^{*}\times\C^{n-1}$.
\end{itemize}
Recall that every rectifiable arc in $\C^{n}$ is polynomially convex (cf. \cite[Corollary 3.1.2]{Stout2007}).
Thus $C\setminus\{q\}\subset\C^{*}\times\C^{n-1}$ is a closed $\cO(\C^{*}\times\C^{n-1})$-convex set since it is exhausted by rectifiable arcs.
Consider the exponential map
\begin{align*}
\pi:\C\times\C^{n-1}\to\C^{*}\times\C^{n-1},\quad\pi(z,w)=(\exp z,w).
\end{align*}
Note that the inverse image $S=\pi^{-1}(C\setminus\{q\})\subset\C^{n}$ is polynomially convex and satisfies the assumption in Theorem \ref{theorem:tame}.
Thus Theorem \ref{theorem:tame} implies that the complement $\C^{n}\setminus S$ is Oka.
Since
\begin{align*}
\pi|_{\C^{n}\setminus S}:\C^{n}\setminus S\to(\C^{*}\times\C^{n-1})\setminus C
\end{align*}
is a holomorphic covering map, the complement $(\C^{*}\times\C^{n-1})\setminus C$ is also Oka (cf. \cite[Proposition 5.6.3]{Forstneric2017}).
Since it is a Zariski open Oka neighborhood of $p$, the localization principle (Theorem \ref{theorem:localization}) implies that the complement $\C^{n}\setminus C$ is Oka.
\end{proof}

Note that Theorem \ref{theorem:main} and the fact used in the above proof also imply the following.

\begin{corollary}
For any rectifiable arc $C\subset\C^{n}$ $(n>1)$, the complement $\C^{n}\setminus C$ is Oka.
\end{corollary}

It is known that every closed algebraic subvariety $S\subset\C^{n}$ of codimension at least two satisfies the assumption in Theorem \ref{theorem:tame} (cf. \cite[\S7.4]{Chirka1989}).
Thus it follows that the complement $\C^{n}\setminus S$ is Oka, which was already observed by Gromov \cite[\S0.5.B]{Gromov1989}.
Since every closed complex subvariety of a Stein space admits a basis of closed holomorphically convex neighborhoods (cf. \cite[Proposition 2.1]{Coltoiu1986}), we can also obtain the following corollary by Theorem \ref{theorem:tame}.

\begin{corollary}
Every closed algebraic subvariety $S\subset\C^{n}$ of codimension at least two admits a basis of closed neighborhoods whose complements are Oka.
\end{corollary}

For a discrete set in $\C^{n}$ $(n\geq 3)$, the assumption in Theorem \ref{theorem:tame} is equivalent to tameness introduced by Rosay and Rudin \cite{Rosay1988} (see \cite[Theorem 3.5]{Rosay1988}).
A discrete set $D\subset\C^{n}$ is said to be \emph{tame} if there exists a holomorphic automorphism $\varphi$ of $\C^{n}$ such that $\varphi(D)=\Z\times\{0\}^{n-1}$ (cf. \cite[Remarks 3.4]{Rosay1988}).
The Oka property of the complement $\C^{n}\setminus D$ of a tame discrete set $D$ was proved by Forstneri\v{c} and Prezelj \cite[Theorem 1.6]{Forstneric2002}.
For the same reason as above, the following corollary holds.

\begin{corollary}
\label{corollary:tame_neighborhood}
Every tame discrete set $D\subset\C^{n}$ $(n\geq3)$ admits a basis of closed neighborhoods whose complements are Oka.
\end{corollary}

Related to Corollary \ref{corollary:tame_neighborhood}, Buzzard \cite{Buzzard2003} studied dominability of the complement of the $\varepsilon$-neighborhood of a tame discrete set.
His results, the question of Rosay and Rudin \cite[Question 3]{Rosay1988} and Corollary \ref{corollary:tame_neighborhood} lead us to the following question.

\begin{question}
Assume that $\varepsilon>0$ and $D\subset\C^{n}$ $(n>1)$ is a discrete set such that
\begin{align*}
\inf_{p,q\in D,\ p\neq q}\|p-q\|>2\varepsilon.
\end{align*}
Does it follow that the complement $\C^{n}\setminus\bigcup_{p\in D}\overline{\B^{n}(p,\varepsilon)}$ is Oka?
Here, $\overline{\B^{n}(p,\varepsilon)}\subset\C^{n}$ denotes the closed ball of radius $\varepsilon$ centered at $p$.
\end{question}

For the standard tame discrete sets, we can give the following positive answer as an application of Theorem \ref{theorem:main}.

\begin{corollary}
For any $n>1$ and any $0<\varepsilon<1/2$, the complement $\C^{n}\setminus\bigcup_{p\in\Z\times\{0\}^{n-1}}\overline{\B^{n}(p,\varepsilon)}$ is Oka.
\end{corollary}

\begin{proof}
Recall that $\C^{*}\times\C^{n-1}$ enjoys the density property by the result of Varolin \cite[Theorem 4.2]{Varolin2001}.
Let us consider the exponential map
\begin{align*}
\pi:\C\times\C^{n-1}\to\C^{*}\times\C^{n-1},\quad\pi(z,w)=(\exp(2\pi iz),w).
\end{align*}
Note that there exists an open polydisc $D\subset\C^{n}$ such that $\overline{\B^{n}(0,\varepsilon)}\subset D$, the restriction $\pi|_{D}:D\to\pi(D)$ is biholomorphic and $\pi(D)\subset\C^{*}\times\C^{n-1}$ is Runge.
Thus the image $\pi(\overline{\B^{n}(0,\varepsilon)})\subset\C^{*}\times\C^{n-1}$ is a compact $\cO(\C^{*}\times\C^{n-1})$-convex set.
Then it follows that the complement $(\C^{*}\times\C^{n-1})\setminus\pi(\overline{\B^{n}(0,\varepsilon)})$ is Oka from Theorem \ref{theorem:main}.
This implies that its universal covering
\begin{align*}
\C^{n}\setminus\bigcup_{p\in\Z\times\{0\}^{n-1}}\overline{\B^{n}(p,\varepsilon)}=\C^{n}\setminus\pi^{-1}\left(\pi\left(\overline{\B^{n}(0,\varepsilon)}\right)\right)
\end{align*}
is also Oka (cf. \cite[Proposition 5.6.3]{Forstneric2017}).
\end{proof}

In the rest, we consider graph complements (see \cite[\S5.4]{Kusakabe2021b} for related results).
The following is an immediate consequence of Theorem \ref{theorem:main_relative}.

\begin{corollary}
\label{corollary:graph_complement}
Assume that $Y$ is a Stein space, $\pi:Y\to B$ is a holomorphic submersion with the density property and $f_{j}:B\to Y$ $(j=1,\ldots,k)$ are holomorphic sections of $\pi$.
Then the restriction $\pi|_{Y\setminus\bigcup_{j=1}^{k}f_{j}(B)}:Y\setminus\bigcup_{j=1}^{k}f_{j}(B)\to B$ enjoys the Oka property.
\end{corollary}

In particular, Corollary \ref{corollary:graph_complement} implies the following Oka principle.

\begin{corollary}
Let $\pi:Y\to B$ be a holomorphic submersion between Stein spaces, $f_{j}:B\to Y$ $(j=1,\ldots,k)$ be holomorphic sections of $\pi$ and $\varphi_{0}:B\to Y$ be a continuous section of $\pi$.
Assume that $\pi$ enjoys the density property and $\varphi_{0}(b)\neq f_{j}(b)$ for all $j=1,\ldots,k$ and $b\in B$.
Then there exists a homotopy $\varphi_{t}:B\to Y$ $(t\in[0,1])$ of continuous sections of $\pi$ such that
\begin{enumerate}
\item $\varphi_{t}(b)\neq f_{j}(b)$ for all $j=1,\ldots,k$, $b\in B$ and $t\in[0,1]$, and
\item $\varphi_{1}:B\to Y$ is a holomorphic section.
\end{enumerate}
\end{corollary}

For a complex manifold $Y$, let us consider the \emph{configuration space} of ordered $n$-tuples of points in $Y$: $F(Y,n)=\{(y_{1},\ldots,y_{n})\in Y^{n}:y_{j}\neq y_{k}\ \text{if}\ j\neq k\}$.
Kutzschebauch and Ramos-Peon proved that the configuration spaces $F(Y,n)$ $(n\in\N)$ of a Stein manifold $Y$ with the density property are Oka \cite[Theorem 3.1]{Kutzschebauch2017}.
Recall that if a surjective holomorphic Serre fibration $\pi:E\to B$ between complex manifolds enjoys the Oka property, then $E$ is an Oka manifold if and only if $B$ is an Oka manifold (cf. \cite[Corollary 2.51]{Forstneric2013}).
Corollary \ref{corollary:graph_complement} implies the following stronger result.

\begin{corollary}
Assume that $Y$ is a Stein manifold with the density property.
Then for each $n\in\N$ the projection
\begin{align*}
F(Y,n+1)\to F(Y,n),\quad (y_{1},\ldots,y_{n+1})\mapsto(y_{1},\ldots,y_{n})
\end{align*}
enjoys the Oka property.
\end{corollary}

Recall that for a continuous function $f:\D\to\C$, the projection $(\D\times\C)\setminus\Gamma_{f}\to\D$ from the complement of the graph $\Gamma_{f}$ of $f$ enjoys the Oka property if and only if $f$ is holomorphic (cf. \cite[Corollary 7.4.10]{Forstneric2017}).
By Theorem \ref{theorem:main_relative}, however, there exists a non-holomorphic real analytic map $f:\D\to\C^{n}$ $(n>1)$ such that the projection $(\D\times\C^{n})\setminus\Gamma_{f}\to\D$ enjoys the Oka property (e.g. $f(z)=(\bar{z},0,\ldots,0)$).
This phenomenon leads us to the following question.

\begin{question}
Assume that $f:\D\to\C^{n}$ $(n>1)$ is a continuous map.
Is there a characterization of the Oka property of the projection $(\D\times\C)\setminus\Gamma_{f}\to\D$ by some function-theoretic property of $f$?
\end{question}

%
%

\section*{Acknowledgement}
I wish to thank Katsutoshi Yamanoi and Franc Forstneri\v{c} for many helpful comments and suggestions.
I am also grateful to the anonymous referee for a careful reading of the manuscript and thoughtful comments which improved the paper.
This work was supported by JSPS KAKENHI Grant Number JP18J20418.

%
%

\end{document}